\documentclass[secthm,seceqn,amsthm,ussrhead,12pt]{amsart}
\usepackage[utf8]{inputenc}
\usepackage[english]{babel}
\usepackage{amssymb,amsmath,amsthm,amsfonts,xcolor,enumerate,hyperref,comment,longtable,cleveref}

\usepackage{times}
\usepackage{cite}
\usepackage{pdflscape}
\usepackage{ulem}
\usepackage[mathcal]{euscript}
\usepackage{tikz}
\usepackage{hyperref}
\usepackage{cancel}
\usepackage{stmaryrd}

\usetikzlibrary{arrows}

\newtheorem{theorem}{Theorem}
\newtheorem{lemma}[theorem]{Lemma}

\newtheorem{remark}[theorem]{Remark}

\newtheorem{definition}[theorem]{Definition}

\newenvironment{Proof}[1][Proof.]{\begin{trivlist}
\item[\hskip \labelsep {\bfseries #1}]}{\flushright
$\Box$\end{trivlist}}

\usepackage{stmaryrd}
\usepackage{xcolor}

\begin{document}
\noindent{\Large
The algebraic and geometric classification of nilpotent Novikov algebras}
\footnote{
This work was supported by Agencia Estatal de Investigación (Spain), grant MTM2016-79661-P (European FEDER support included, UE); 
IMU/CDC Grant (Abel Visiting Scholar Program);  RFBR 18-31-20004; FAPESP  18/15712-0.
The authors  would like to acknowledge the hospitality of the University of Santiago de Compostela (Spain).}

   \

   {\bf  Iqboljon Karimjanov$^{a}$,
   Ivan   Kaygorodov$^{b}$ \&
   Abror Khudoyberdiyev$^{c}$}

\

 \

{\tiny

$^{a}$ University of Santiago de Compostela, Santiago de Compostela, Spain.

$^{b}$ CMCC, Universidade Federal do ABC, Santo Andr\'e, Brazil.

$^{c}$ National University of Uzbekistan, Institute of Mathematics Academy of
Sciences of Uzbekistan, Tashkent, Uzbekistan.

\smallskip

   E-mail addresses:

\smallskip

Iqboljon Karimjanov (iqboli@gmail.com)

Ivan   Kaygorodov (kaygorodov.ivan@gmail.com)

Abror Khudoyberdiyev (khabror@mail.ru)

}

\

\

\noindent{\bf Abstract}:
{\it This paper is devoted to give the complete algebraic and geometric classification of $4$-dimensional nilpotent Novikov
 algebras over $\mathbb C.$}

\


\


\section*{Introduction}

The algebraic classification (up to isomorphism) of algebras of dimension $n$ from a certain variety
defined by some family of polinomial identities is a classic problem in the theory of non-associative algebras.
There are many results related to the algebraic classification of algebras with small dimensional in the varieties of
Jordan, Lie, Leibniz, Zinbiel and many other algebras \cite{ack, cfk19, ck13, gkks, degr3, usefi1, degr1, degr2, ha16, hac18, kv16, krh14}.
Another interesting direction in the classification of algebras is geometric classification.
There are many results related to geometric classification of
Jordan, Lie, Leibniz, Zinbiel and many other algebras
\cite{casas, maria,bb14, BC99, cfk19, GRH, GRH2,gkks, ikv17, ikv18, kppv, kpv, kv16, S90}.
In the present paper, we give the algebraic  and geometric classification of
$4$-dimensional nilpotent Novikov  algebras
introduced by Novikov and Balinsky in \cite{bn85}.

The variety of Novikov algebras is defined by the following identities:
\[
\begin{array}{rcl}
(xy)z &=& (xz)y,\\
(xy)z-x(yz) &=& (yx)z-y(xz).
\end{array} \]
It contains the commutative associative algebras as a subvariety.
On the other hand, the variety of Novikov algebras is the intersection of
the variety of right commutative algebras (defined by the first Novikov identity)
and
the variety of left symmetric (Pre-Lie) algebras
(defined by the second Novikov identity).
Also, any Novikov algebra with the commutator multiplication gives a Lie algebra,
and Novikov algebras are related to
Tortken and
Novikov-Poisson algebras  \cite{dz02, xu97}.
The sistematic study of Novikov algebras started after the paper of Zelmanov where
all finite-dimensional simple Novikov algebras over the complex field were found \cite{ze87}.
The first nontrivial examples of  infinite-dimensional simple  Novikov algebras were constructed in \cite{fi89}.
Also, the simple Novikov algebras were described in infinite-dimensional case and over fields of positive characteristic \cite{os94, xu96, xu01}.
The algebraic classification of $3$-dimensional Novikov algebras was given in \cite{bc01},
and for some classes of $4$-dimensional algebras, it was given in \cite{bg13};
the geometric classification of $3$-dimensional Novikov algebras was given in \cite{bb14}.
Many other pure algebraic properties were studied in a series of papers of Dzhumadildaev \cite{di14, dt05, dz11, dl02}.

Our method for classifying nilpotent Novikov algebras is based on the calculation of central extensions of smaller nilpotent algebras from the same variety.
Central extensions play an important role in quantum mechanics: one of the earlier
encounters is through Wigner\'{}s theorem, which states that a symmetry of a quantum
mechanical system determines an (anti-)unitary transformation of a Hilbert space.
Another area of physics where one encounters central extensions is the quantum theory
of conserved currents of a Lagrangian. These currents span an algebra which is closely
related to the so-called affine Kac--Moody algebras, the universal central extensions
of loop algebras.
Central extensions are needed in physics because the symmetry group of a quantized
system is usually a central extension of the classical symmetry group, and in the same way
the corresponding symmetry Lie algebra of the quantum system is, in general, a central
extension of the classical symmetry algebra. Kac--Moody algebras have been conjectured
to be the symmetry groups of a unified superstring theory. The centrally extended Lie
algebras play a dominant role in quantum field theory, particularly in conformal field
theory, string theory and in $M$-theory.
In the theory of Lie groups, Lie algebras and their representations, a Lie algebra extension
is an enlargement of a given Lie algebra $g$ by another Lie algebra $h.$ Extensions
arise in several ways. For example, the trivial extension is obtained as a direct sum of
two Lie algebras. Other types are split and central extensions. Extensions
arise naturally, for instance, when constructing a Lie algebra from projective group representations.
A central extension and an extension by a derivation of a polynomial loop algebra
over a finite-dimensional simple Lie algebra give a Lie algebra isomorphic to a
non-twisted affine Kac--Moody algebra \cite[Chapter 19]{bkk}. 
The algebraic study of central extensions of Lie and non-Lie algebras has been an important topic for years \cite{omirov,ha17,hac16,kkl18,ss78,zusmanovich}.
First, Skjelbred and Sund used central extensions of Lie algebras for a classification of nilpotent Lie algebras  \cite{ss78}.
After that, using the method described by Skjelbred and Sund,  all non-Lie central extensions of  all $4$-dimensional Malcev algebras were described \cite{hac16}, and also
all the non-associative central extensions of $3$-dimensional Jordan algebras \cite{ha17},
all the anticommutative central extensions of the $3$-dimensional anticommutative algebras \cite{cfk182},
and all the central extensions of the $2$-dimensional algebras \cite{cfk18}.
Note that the method of central extensions is an important tool in the classification of nilpotent algebras
(see, for example, \cite{ha16n}),
which was used to describe
all the $4$-dimensional nilpotent associative algebras \cite{degr1},
all the $4$-dimensional nilpotent bicommutative algebras \cite{kpv19},
all the $5$-dimensional nilpotent Jordan algebras \cite{ha16},
all the $5$-dimensional nilpotent restricted Lie algebras \cite{usefi1},
all the $6$-dimensional nilpotent Lie algebras \cite{degr3,degr2},
all the $6$-dimensional nilpotent Malcev algebras \cite{hac18}
and some others.

$\mathfrak{Remark.}$
Note that, the algebraic classification of all $4$-dimension nilpotent Novikov algebras can be found as a corollary from \cite{bg13},
but in our opinion the method used in this paper is not clear to understand and many calculations were omitted.
Our algebraic classification is given by another method and it confirms the result from \cite{bg13}.

\section{The algebraic classification of nilpotent Novikov algebras}
\subsection{Method of classification of nilpotent algebras}

Throughout this paper, we will use the notations and methods well written in \cite{ha17,hac16,cfk18},
which we have adapted for the Novikov case with some modifications.
From now, we will give only some important definitions.

Let $({\bf A}, \cdot)$ be a Novikov  algebra over  $\mathbb C$
and $\mathbb V$ a vector space over ${\mathbb C}$. Then the $\mathbb C$-linear space ${\rm Z^{2}}\left(
\bf A,\mathbb V \right) $ is defined as the set of all  bilinear maps $\theta  \colon {\bf A} \times {\bf A} \longrightarrow {\mathbb V}$,
such that

\[ \theta(xy,z)=\theta(xz,y), \]
\[ \theta(xy,z)-\theta(x,yz)= \theta(yx,z)-\theta(y,xz). \]
These elements will be called {\it cocycles}. For a
linear map $f$ from $\bf A$ to  $\mathbb V$, if we define $\delta f\colon {\bf A} \times
{\bf A} \longrightarrow {\mathbb V}$ by $\delta f  (x,y ) =f(xy )$, then $\delta f\in {\rm Z^{2}}\left( {\bf A},{\mathbb V} \right) $. We define ${\rm B^{2}}\left({\bf A},{\mathbb V}\right) =\left\{ \theta =\delta f\ : f\in {\rm Hom}\left( {\bf A},{\mathbb V}\right) \right\} $.
One can easily check that ${\rm B^{2}}(\bf A,\mathbb V)$ is a linear subspace of ${\rm Z^{2}}\left( {\bf A},{\mathbb V}\right) $; its elements are called
{\it coboundaries}. We define the {\it second cohomology space} ${\rm H^{2}}\left( {\bf A},{\mathbb V}\right) $ as the quotient space ${\rm Z^{2}}
\left( {\bf A},{\mathbb V}\right) \big/{\rm B^{2}}\left( {\bf A},{\mathbb V}\right) $.

\

Let $\operatorname{Aut}({\bf A}) $ be the automorphism group of the Novikov algebra ${\bf A} $ and let $\phi \in \operatorname{Aut}({\bf A})$. For $\theta \in
{\rm Z^{2}}\left( {\bf A},{\mathbb V}\right) $ define $\phi \theta (x,y)
=\theta \left( \phi \left( x\right) ,\phi \left( y\right) \right) $. Then $\phi \theta \in {\rm Z^{2}}\left( {\bf A},{\mathbb V}\right) $. So, $\operatorname{Aut}({\bf A})$
acts on ${\rm Z^{2}}\left( {\bf A},{\mathbb V}\right) $. It is easy to verify that
 ${\rm B^{2}}\left( {\bf A},{\mathbb V}\right) $ is invariant under the action of $\operatorname{Aut}({\bf A}).$  
 So, we have that $\operatorname{Aut}({\bf A})$ acts on ${\rm H^{2}}\left( {\bf A},{\mathbb V}\right)$.

\

Let $\bf A$ be a Novikov  algebra of dimension $m<n$ over  $\mathbb C$ and ${\mathbb V}$ a $\mathbb C$-vector
space of dimension $n-m$. For any $\theta \in {\rm Z^{2}}\left(
{\bf A},{\mathbb V}\right) $, define on the linear space ${\bf A}_{\theta } = {\bf A}\oplus {\mathbb V}$ the
bilinear product `` $\left[ -,-\right] _{{\bf A}_{\theta }}$'' by $\left[ x+x^{\prime },y+y^{\prime }\right] _{{\bf A}_{\theta }}=
 xy +\theta(x,y) $ for all $x,y\in {\bf A},x^{\prime },y^{\prime }\in {\mathbb V}$.
The algebra ${\bf A}_{\theta }$ is a Novikov algebra which is called an $(n-m)$-{\it dimensional central extension} of ${\bf A}$ by ${\mathbb V}$. Indeed, we have, in a straightforward way, that ${\bf A_{\theta}}$ is a Novikov algebra if and only if $\theta \in {\rm Z^2}({\bf A}, {\mathbb V})$.

We also call the
set $\operatorname{Ann}(\theta)=\left\{ x\in {\bf A}:\theta \left( x, {\bf A} \right)+ \theta \left({\bf A} ,x\right) =0\right\} $
the {\it annihilator} of $\theta $. We recall that the {\it annihilator} of an  algebra ${\bf A}$ is defined as
the ideal $\operatorname{Ann}(  {\bf A} ) =\left\{ x\in {\bf A}:  x{\bf A}+ {\bf A}x =0\right\}$. Observe
 that
$\operatorname{Ann}\left( {\bf A}_{\theta }\right) =\operatorname{Ann}(\theta) \cap\operatorname{Ann}({\bf A})
 \oplus {\mathbb V}$.

\

We have the following  key result:

\begin{lemma}
Let ${\bf A}$ be an $n$-dimensional Novikov algebra such that $\dim (\operatorname{Ann}({\bf A}))=m\neq0$. Then there exists, up to isomorphism, a unique $(n-m)$-dimensional Novikov  algebra ${\bf A}'$ and a bilinear map $\theta \in {\rm Z^2}({\bf A}, {\mathbb V})$ with $\operatorname{Ann}({\bf A})\cap\operatorname{Ann}(\theta)=0$, where $\mathbb V$ is a vector space of dimension m, such that ${\bf A} \cong {{\bf A}'}_{\theta}$ and
 ${\bf A}/\operatorname{Ann}({\bf A})\cong {\bf A}'$.
\end{lemma}

\begin{proof}
Let ${\bf A}'$ be a linear complement of $\operatorname{Ann}({\bf A})$ in ${\bf A}$. Define a linear map $P \colon {\bf A} \longrightarrow {\bf A}'$ by $P(x+v)=x$ for $x\in {\bf A}'$ and $v\in\operatorname{Ann}({\bf A})$, and define a multiplication on ${\bf A}'$ by $[x, y]_{{\bf A}'}=P(x y)$ for $x, y \in {\bf A}'$.
For $x, y \in {\bf A}$, we have
\[P(xy)=P((x-P(x)+P(x))(y- P(y)+P(y)))=P(P(x) P(y))=[P(x), P(y)]_{{\bf A}'}. \]

Since $P$ is a homomorphism we have that $P({\bf A})={\bf A}'$ is a Novikov algebra and
 ${\bf A}/\operatorname{Ann}({\bf A})\cong {\bf A}'$, which gives us the uniqueness. Now, define the map $\theta \colon {\bf A}' \times {\bf A}' \longrightarrow\operatorname{Ann}({\bf A})$ by $\theta(x,y)=xy- [x,y]_{{\bf A}'}$.
  Thus, ${\bf A}'_{\theta}$ is ${\bf A}$ and therefore $\theta \in {\rm Z^2}({\bf A}, {\mathbb V})$ and $\operatorname{Ann}({\bf A})\cap\operatorname{Ann}(\theta)=0$.
\end{proof}

\

\;
\begin{definition}
Given an algebra ${\bf A}$, if ${\bf A}=I\oplus \mathbb Cx$
is a direct sum of two ideals, i.e., $x\in \operatorname{Ann}({\bf A})$, then $\mathbb Cx$ is called an {\it annihilator component} of ${\bf A}$.
\end{definition}
\begin{definition}
A central extension of an algebra $\bf A$ without annihilator component is called a non-split central extension.
\end{definition}

However, in order to solve the isomorphism problem we need to study the
action of $\operatorname{Aut}({\bf A})$ on ${\rm H^{2}}\left( {\bf A},{\mathbb V}
\right) $. To do that, let us fix a basis $e_{1},\ldots ,e_{s}$ of ${\mathbb V}$, and $
\theta \in {\rm Z^{2}}\left( {\bf A},{\mathbb V}\right) $. Then $\theta $ can be uniquely
written as $\theta \left( x,y\right) =
\displaystyle \sum_{i=1}^{s} \theta _{i}\left( x,y\right) e_{i}$, where $\theta _{i}\in
{\rm Z^{2}}\left( {\bf A},\mathbb C\right) $. Moreover, $\operatorname{Ann}(\theta)=\operatorname{Ann}(\theta _{1})\cap\operatorname{Ann}(\theta _{2})\cdots \cap\operatorname{Ann}(\theta _{s})$. Furthermore, $\theta \in
{\rm B^{2}}\left( {\bf A},{\mathbb V}\right) $\ if and only if all $\theta _{i}\in {\rm B^{2}}\left( {\bf A},
\mathbb C\right) $.
It is not difficult to prove (see \cite[Lemma 13]{hac16}), that given a Novikov algebra ${\bf A}_{\theta}$, if we write as
above $\theta \left( x,y\right) = \displaystyle \sum_{i=1}^{s} \theta_{i}\left( x,y\right) e_{i}\in {\rm Z^{2}}\left( {\bf A},{\mathbb V}\right) $ and we have
$\operatorname{Ann}(\theta)\cap \operatorname{Ann}\left( {\bf A}\right) =0$, then ${\bf A}_{\theta }$ has an
annihilator component if and only if $\left[ \theta _{1}\right] ,\left[
\theta _{2}\right] ,\ldots ,\left[ \theta _{s}\right] $ are linearly
dependent in ${\rm H^{2}}\left( {\bf A},\mathbb C\right) $.

\;

Let ${\mathbb V}$ be a finite-dimensional vector space over $\mathbb C$. The {\it Grassmannian} $G_{k}\left( {\mathbb V}\right) $ is the set of all $k$-dimensional
linear subspaces of $ {\mathbb V}$. Let $G_{s}\left( {\rm H^{2}}\left( {\bf A},\mathbb C\right) \right) $ be the Grassmannian of subspaces of dimension $s$ in
${\rm H^{2}}\left( {\bf A},\mathbb C\right) $. There is a natural action of $\operatorname{Aut}({\bf A})$ on $G_{s}\left( {\rm H^{2}}\left( {\bf A},\mathbb C\right) \right) $.
Let $\phi \in \operatorname{Aut}({\bf A})$. For $W=\left\langle
\left[ \theta _{1}\right] ,\left[ \theta _{2}\right] ,\dots,\left[ \theta _{s}
\right] \right\rangle \in G_{s}\left( {\rm H^{2}}\left( {\bf A},\mathbb C
\right) \right) $ define $\phi W=\left\langle \left[ \phi \theta _{1}\right]
,\left[ \phi \theta _{2}\right] ,\dots,\left[ \phi \theta _{s}\right]
\right\rangle $. Then $\phi W\in G_{s}\left( {\rm H^{2}}\left( {\bf A},\mathbb C \right) \right) $. We denote the orbit of $W\in G_{s}\left(
{\rm H^{2}}\left( {\bf A},\mathbb C\right) \right) $ under the action of $\operatorname{Aut}({\bf A})$ by $\operatorname{Orb}(W)$. Given
\[
W_{1}=\left\langle \left[ \theta _{1}\right] ,\left[ \theta _{2}\right] ,\dots,
\left[ \theta _{s}\right] \right\rangle ,W_{2}=\left\langle \left[ \vartheta
_{1}\right] ,\left[ \vartheta _{2}\right] ,\dots,\left[ \vartheta _{s}\right]
\right\rangle \in G_{s}\left( {\rm H^{2}}\left( {\bf A},\mathbb C\right)
\right),
\]
we easily have that in case $W_{1}=W_{2}$, then $ \bigcap\limits_{i=1}^{s}\operatorname{Ann}(\theta _{i})\cap \operatorname{Ann}\left( {\bf A}\right) = \bigcap\limits_{i=1}^{s}
\operatorname{Ann}(\vartheta _{i})\cap\operatorname{Ann}( {\bf A}) $, and therefore we can introduce
the set
\[
{\bf T}_{s}({\bf A}) =\left\{ W=\left\langle \left[ \theta _{1}\right] ,
\left[ \theta _{2}\right] ,\dots,\left[ \theta _{s}\right] \right\rangle \in
G_{s}\left( {\rm H^{2}}\left( {\bf A},\mathbb C\right) \right) : \bigcap\limits_{i=1}^{s}\operatorname{Ann}(\theta _{i})\cap\operatorname{Ann}({\bf A}) =0\right\},
\]
which is stable under the action of $\operatorname{Aut}({\bf A})$.

\

Now, let ${\mathbb V}$ be an $s$-dimensional linear space and let us denote by
${\bf E}\left( {\bf A},{\mathbb V}\right) $ the set of all {\it non-split $s$-dimensional central extensions} of ${\bf A}$ by
${\mathbb V}$. We can write
\[
{\bf E}\left( {\bf A},{\mathbb V}\right) =\left\{ {\bf A}_{\theta }:\theta \left( x,y\right) = \sum_{i=1}^{s}\theta _{i}\left( x,y\right) e_{i} \ \ \text{and} \ \ \left\langle \left[ \theta _{1}\right] ,\left[ \theta _{2}\right] ,\dots,
\left[ \theta _{s}\right] \right\rangle \in {\bf T}_{s}({\bf A}) \right\} .
\]
We also have the following result, which can be proved as \cite[Lemma 17]{hac16}.

\begin{lemma}
 Let ${\bf A}_{\theta },{\bf A}_{\vartheta }\in {\bf E}\left( {\bf A},{\mathbb V}\right) $. Suppose that $\theta \left( x,y\right) =  \displaystyle \sum_{i=1}^{s}
\theta _{i}\left( x,y\right) e_{i}$ and $\vartheta \left( x,y\right) =
\displaystyle \sum_{i=1}^{s} \vartheta _{i}\left( x,y\right) e_{i}$.
Then the Novikov algebras ${\bf A}_{\theta }$ and ${\bf A}_{\vartheta } $ are isomorphic
if and only if
$$\operatorname{Orb}\left\langle \left[ \theta _{1}\right] ,
\left[ \theta _{2}\right] ,\dots,\left[ \theta _{s}\right] \right\rangle =
\operatorname{Orb}\left\langle \left[ \vartheta _{1}\right] ,\left[ \vartheta
_{2}\right] ,\dots,\left[ \vartheta _{s}\right] \right\rangle .$$
\end{lemma}

This shows that there exists a one-to-one correspondence between the set of $\operatorname{Aut}({\bf A})$-orbits on ${\bf T}_{s}\left( {\bf A}\right) $ and the set of
isomorphism classes of ${\bf E}\left( {\bf A},{\mathbb V}\right) $. Consequently we have a
procedure that allows us, given a Novikov algebra ${\bf A}'$ of
dimension $n-s$, to construct all non-split central extensions of ${\bf A}'$. This procedure is:

\; \;

{\centerline {\textsl{Procedure}}}

\begin{enumerate}
\item For a given Novikov algebra ${\bf A}'$ of dimension $n-s $, determine ${\rm H^{2}}( {\bf A}',\mathbb {C}) $, $\operatorname{Ann}({\bf A}')$ and $\operatorname{Aut}({\bf A}')$.

\item Determine the set of $\operatorname{Aut}({\bf A}')$-orbits on $T_{s}({\bf A}') $.

\item For each orbit, construct the Novikov algebra associated with a
representative of it.
\end{enumerate}

\

\subsection{Notations}
Let ${\bf A}$ be a Novikov algebra with
a basis $e_{1},e_{2},\dots,e_{n}$. Then by $\Delta _{ij}$\ we will denote the
bilinear form
$\Delta _{ij} \colon {\bf A}\times {\bf A}\longrightarrow \mathbb C$
with $\Delta _{ij}\left( e_{l},e_{m}\right) = \delta_{il}\delta_{jm}$.
Then the set $\left\{ \Delta_{ij}:1\leq i, j\leq n\right\} $ is a basis for the linear space of
the bilinear forms on ${\bf A}$. Then every $\theta \in
{\rm Z^{2}}\left( {\bf A},\mathbb C\right) $ can be uniquely written as $
\theta = \displaystyle \sum_{1\leq i,j\leq n} c_{ij}\Delta _{{i}{j}}$, where $
c_{ij}\in \mathbb C$.
Let us fix the following notations:

$$\begin{array}{lll}
{\mathcal N}^{i*}_j& \mbox{---}& j\mbox{th }i\mbox{-dimensional nilpotent  Novikov algebra with identity $xyz=0$} \\
{\mathcal N}^i_j& \mbox{---}& j\mbox{th }i\mbox{-dimensional nilpotent "pure" Novikov algebra (without identity $xyz=0$)} \\
{\mathfrak{N}}_i& \mbox{---}& i\mbox{-dimensional algebra with zero product} \\
({\bf A})_{i,j}& \mbox{---}& j\mbox{th }i\mbox{-dimensional central extension of }\bf A. \\
\end{array}$$

\subsection{The algebraic classification of  $3$-dimensional nilpotent Novikov algebras}
There are no nontrivial $1$-dimensional nilpotent Novikov algebras.
There is only one nontrivial $2$-dimensional nilpotent Novikov algebra
(it is the non-split central extension of $1$-dimensional algebra with zero product):

$$\begin{array}{ll llll}
{\mathcal N}^{2*}_{01} &:& (\mathfrak{N}_1)_{2,1} &:& e_1 e_1 = e_2.\\
\end{array}$$

Thanks to \cite{cfk18} we have the description of all central extensions of  ${\mathcal N}^{2*}_{01}$ and $\mathfrak{N}_2$.
Choosing the Novikov algebras between the central extensions of these algebras,
we have the classification of all non-split $3$-dimensional nilpotent Novikov algebras:

$$\begin{array}{ll llllllllllll}
{\mathcal N}^{3*}_{02} &:& (\mathfrak{N}_2)_{3,1} &:& e_1 e_1 = e_3, &  e_2 e_2=e_3; \\
{\mathcal N}^{3*}_{03} &:& (\mathfrak{N}_2)_{3,2} &:& e_1 e_2=e_3, & e_2 e_1=-e_3;   \\
{\mathcal N}^{3*}_{04}(\lambda) &:& (\mathfrak{N}_2)_{3,3} &:&
e_1 e_1 = \lambda e_3,  & e_2 e_1=e_3,  & e_2 e_2=e_3, & \lambda \neq 0; \\
{\mathcal N}^{3*}_{04}(0) &:& (\mathfrak{N}_2)_{3,3} &:&   e_1 e_2=e_3,  \\
{\mathcal N}^3_{01} &:& ({\mathcal N}^{2*}_{01} )_{3,1} &:& e_1 e_1 = e_2,  & e_2 e_1=e_3;  \\
{\mathcal N}^3_{02}(\lambda) &:& ({\mathcal N}^{2*}_{01} )_{3,2} &:& e_1 e_1 = e_2, & e_1 e_2=e_3, & e_2 e_1=\lambda e_3, & \lambda\in {\mathbb C}.
\end{array} $$

\subsection{$1$-dimensional central extensions of $3$-dimensional  nilpotent Novikov algebras}
\label{centrext}
\subsubsection{The description of second cohomology space of  $3$-dimensional nilpotent Novikov algebras}

\
In the following table we give the description of the second cohomology space of  $3$-dimensional nilpotent Novikov algebras

{\tiny
$$
\begin{array}{|l|l|l|l|}
\hline
\bf A  & {\rm Z^{2}}\left( {\bf A}\right)  & {\rm B^2}({\bf A}) & {\rm H^2}({\bf A}) \\
\hline
\hline
{\mathcal N}^{3*}_{01} &  \langle
\Delta_{11},\Delta_{12}, \Delta_{13}, \Delta_{21},\Delta_{31}, \Delta_{33}\rangle
&\langle \Delta_{11} \rangle&
\langle[\Delta_{12}], [\Delta_{13}], [\Delta_{21}], [\Delta_{31}], [\Delta_{33}] \rangle\\
\hline

{\mathcal N}^{3*}_{02} &  \langle  \Delta_{11},\Delta_{12}, \Delta_{21}, \Delta_{22} \rangle
& \langle\Delta_{11}+\Delta_{22}\rangle &  \langle [\Delta_{12}], [\Delta_{21}], [\Delta_{22}] \rangle \\
\hline
{\mathcal N}^{3*}_{03} & \langle  \Delta_{11},\Delta_{12}, \Delta_{21}, \Delta_{22} \rangle
& \langle\Delta_{12} -\Delta_{21} \rangle&  \langle [\Delta_{11}], [\Delta_{21}], [\Delta_{22}] \rangle \\
\hline
{\mathcal N}^{3*}_{04}(\lambda)_{\lambda\neq 0} & \langle  \Delta_{11},\Delta_{12}, \Delta_{21}, \Delta_{22} \rangle & \langle\lambda\Delta_{11}+\Delta_{21}+\Delta_{22}\rangle &  \langle [\Delta_{12}], [\Delta_{21}], [\Delta_{22}] \rangle \\

\hline
{\mathcal N}^{3*}_{04}(0) &
\Big\langle
\begin{array}{l} \Delta_{11},\Delta_{12},  \Delta_{13}, \\
\Delta_{21}, \Delta_{22}, \Delta_{23}-\Delta_{32}
\end{array}

\Big\rangle & \langle \Delta_{12} \rangle &

\Big\langle
\begin{array}{l}
[\Delta_{11}], [\Delta_{13}], [\Delta_{21}], \\

[\Delta_{22}],   [\Delta_{23}]-[\Delta_{32}]

\end{array} \Big\rangle

\\

\hline
{\mathcal N}^3_{01} &\langle \Delta_{11},\Delta_{12}, \Delta_{21},\Delta_{13}-\Delta_{31}\rangle& \langle \Delta_{11}, \Delta_{21}\rangle &  \langle [\Delta_{12}], [\Delta_{13}]-[\Delta_{31}] \rangle \\
\hline
{\mathcal N}^3_{02}(\lambda) &\Big\langle
\begin{array}{l}\Delta_{11},\Delta_{12}, \Delta_{21},\\ (2-\lambda)\Delta_{13}+\lambda(\Delta_{22}+\Delta_{31}) \end{array} \Big\rangle&
\Big\langle
\begin{array}{l} \Delta_{11},\\  \Delta_{12}+ \lambda \Delta_{21}
\end{array}\Big\rangle &
\Big\langle
\begin{array}{l} [\Delta_{21}], \\
(2-\lambda)[\Delta_{13}]+\lambda([\Delta_{22}]+[\Delta_{31}])
\end{array} \Big\rangle \\

\hline
\end{array}$$
}
where ${\mathcal N}^{3*}_{01}={\mathcal N}^{2*}_{01}\oplus{\mathbb C e_3}.$

\begin{remark}
From the description of cocycles of the algebras ${\mathcal N}^{3*}_{02}, {\mathcal N}^{3*}_{03},
{\mathcal N}^{3*}_{04}(\lambda)_{\lambda\neq0}$, it is not difficult to see that the $1$-dimensional central extensions of these algebras are
 $2$-dimensional central extensions of $2$-dimensional nilpotent Novikov algebras.
 Thanks to  \cite{cfk18} we have the description of all non-split $2$-dimensional central extensions of
 $2$-dimensional nilpotent Novikov algebras:

$$\begin{array}{lllllllll}
{\mathcal N}^4_{03} &:&  ({\mathcal N}^{2*}_{01} )_{4,1} &:& e_1 e_1 = e_2, & e_1 e_2=e_4, & e_2 e_1=e_3.  \\
\end{array}$$

\end{remark}





\subsubsection{Central extensions of ${\mathcal N}^{3*}_{01}$}


Let us use the following notations:
\[
\nabla_1=[\Delta_{12}], \nabla_2=[\Delta_{13}], \nabla_3=[\Delta_{21}], \nabla_4=[\Delta_{31}], \nabla_5=[\Delta_{33}]. \]

The automorphism group of ${\mathcal N}^{3*}_{01}$ consists of invertible matrices of the form

\[\phi=\begin{pmatrix}
x & 0 & 0\\
u & x^2 & w\\
z & 0 & y
\end{pmatrix}. \]

Since

\[ \phi^T\begin{pmatrix}
0 & \alpha_1 & \alpha_2\\
\alpha_3 & 0 & 0\\
\alpha_4 & 0 & \alpha_5
\end{pmatrix}\phi =
\begin{pmatrix}
\alpha^* & x^3 \alpha_1 & w x \alpha_1 + y (x \alpha_2 + z \alpha_5) \\
x^3 \alpha_3 & 0 & 0 \\
x (w \alpha_3 + y \alpha_4) + y z \alpha_5 & 0 & y^2 \alpha_5
\end{pmatrix},
\]
we have that the action of $\operatorname{Aut} (\mathcal{N}_{01}^{3*})$ on subspace
$\Big\langle \sum\limits_{i=1}^5 \alpha_i\nabla_i \Big\rangle$ is given by
$\Big\langle \sum\limits_{i=1}^5 \alpha_i^*\nabla_i \Big\rangle,$
where

\[\begin{array}{rcl}
\alpha^*_1&=&x^3 \alpha_1;\\
\alpha^*_2&=&w x \alpha_1 + y x \alpha_2 + y z \alpha_5;\\
\alpha^*_3&=&x^3 \alpha_3;\\
\alpha^*_4&=& w x \alpha_3 + y x\alpha_4 + y z \alpha_5;\\
\alpha^*_5&=&y^2 \alpha_5.
\end{array}\]

It is easy to see that the elements $\alpha_1 \nabla_1 +\alpha_3\nabla_3$ and  $\alpha_2\nabla_2 +\alpha_4\nabla_4+\alpha_5\nabla_5$ give algebras with $2$-dimensional annihilator, which were found before.
Since we are interested  only in new algebras,  we have the following cases:

\begin{enumerate}
    \item  $\alpha_1\neq0, \alpha_3\neq0, \alpha_5\neq0,$ then:
   \begin{enumerate}
        \item if $\alpha_1\neq\alpha_3,$ then choosing
        $x=\frac{1}{\sqrt[3]{\alpha_1}}, y=\frac{1}{\sqrt{\alpha_5}}, w =\frac{ y (\alpha_2 - \alpha_4)}{\alpha_3- \alpha_1},
        z=\frac{x(\alpha_1\alpha_4 -\alpha_2\alpha_3)}{ \alpha_5(\alpha_3- \alpha_1)}, $
        we have the representative
        $\langle\nabla_1+\alpha \nabla_3+\nabla_5  \rangle_{\alpha\neq0; 1}.$

        \item if $\alpha_1=\alpha_3,$ $\alpha_2\neq\alpha_4,$ then choosing
        $x=\frac{(\alpha_2-\alpha_4)^2}{\alpha_1\alpha_5}, y=\frac{(\alpha_2-\alpha_4)^3}{\alpha_1\alpha_5^2}, w =-\frac{ y x \alpha_4 + y z \alpha_5}{x \alpha_1},
         $
        we have the representative
        $\langle \nabla_1+ \nabla_2+\nabla_3+\nabla_5 \rangle.$

        \item if $\alpha_1=\alpha_3,$ $\alpha_2=\alpha_4,$ then choosing
        $x=\frac{1}{\sqrt[3]{\alpha_1}}, y=\frac{1}{\sqrt{\alpha_5}}, w =-\frac{ y x \alpha_4 + y z \alpha_5}{x \alpha_1},
          $
        we have the representative
        $\langle \nabla_1+\nabla_3+\nabla_5 \rangle.$

    \end{enumerate}
\item $\alpha_1\neq0, \alpha_3\neq0, \alpha_5=0,$ then:
\begin{enumerate}
    \item if $\alpha_1\alpha_4\neq\alpha_2\alpha_3,$ then choosing
    $x=\frac{1}{\sqrt[3]{\alpha_1}}, w = -\frac{y \alpha_4}{ \alpha_3}, y=\frac{x^2\alpha_1}{\alpha_2\alpha_3-\alpha_1\alpha_4},$
    we have the representative
        $\langle \nabla_1+\nabla_2+\alpha \nabla_3 \rangle_{\alpha\neq0}.$
    \item if $\alpha_1\alpha_4=\alpha_2\alpha_3,$  then choosing
        $x=\frac{1}{\sqrt[3]{\alpha_1}}, w = -\frac{y \alpha_4}{ \alpha_3},$
        we have the representative
    $\langle \nabla_1+\alpha \nabla_3 \rangle_{\alpha\neq 0}.$

\end{enumerate}
\item $\alpha_1\neq0, \alpha_3=0, \alpha_5\neq0,$ then choosing
        $x=\frac{1}{\sqrt[3]{\alpha_1}}, w =-\frac{ y x \alpha_2 + y z \alpha_5}{x \alpha_1},
        z=-\frac{ x\alpha_4}{   \alpha_5}, y=\frac{1}{\sqrt{\alpha_5}}, $
        we have the representative $ \langle \nabla_1+\nabla_5 \rangle. $

\item $\alpha_1=0, \alpha_3\neq0, \alpha_5\neq0,$ then choosing
        $x=\frac{1}{\sqrt[3]{\alpha_3}}, z =-\frac{  x \alpha_2}{  \alpha_5},
        w=-\frac{y x\alpha_4+ y z \alpha_5}{ x \alpha_3}, y=\frac{1}{\sqrt{\alpha_5}}, $
        we have the representative $\langle  \nabla_3+\nabla_5 \rangle. $




\item $\alpha_1=0, \alpha_3\neq0, \alpha_5=0,$ then:
\begin{enumerate}
    \item if $\alpha_2\neq0,$ then choosing
    $y=\frac{x^2\alpha_3}{\alpha_2}, x=\frac{1}{\sqrt[3]{\alpha_3}}, w=-\frac{y\alpha_4}{\alpha_3},$
    we have the representative $\langle \nabla_2+\nabla_3 \rangle.$

\item if $\alpha_2=0,$ then choosing
        $x=\frac{1}{\sqrt[3]{\alpha_3}}, w=-\frac{y\alpha_4}{\alpha_3},$
        we have the representative  $\langle \nabla_3 \rangle.$

\end{enumerate}

\item $\alpha_1\neq0, \alpha_3=0, \alpha_5=0,$ then:
    \begin{enumerate}
    \item if $\alpha_4\neq0,$ then choosing
    $x=\frac{1}{\sqrt[3]{\alpha_1}}, w=-\frac{y\alpha_2}{\alpha_1}, y=\frac{x^2\alpha_1}{\alpha_4},$ we have the representative $\langle \nabla_1+\nabla_4 \rangle.$

    \item if $\alpha_4=0,$ then choosing
    $x=\frac{1}{\sqrt[3]{\alpha_1}}, w=-\frac{y\alpha_2}{\alpha_1},$ we have the representative  $\langle \nabla_1 \rangle.$
\end{enumerate}




\end{enumerate}

Now we have all the new $4$-dimensional nilpotent Novikov algebras constructed from   ${\mathcal N}^{3*}_{01}:$
\[ {\mathcal N}^4_{04}(\alpha), \ {\mathcal N}^4_{05}, \ {\mathcal N}^4_{06}(\alpha)_{\alpha\neq 0}, \ {\mathcal N}^4_{07}, \ {\mathcal N}^4_{08}, \ {\mathcal N}^4_{09}. \]
The multiplication tables of these algebras can be found in Appendix A.

\subsubsection{Central extensions of ${\mathcal N}^{3*}_{04}(0)$}



Let us use the following notations:
\[\nabla_1=[\Delta_{11}], \nabla_2=[\Delta_{13}], \nabla_3=[\Delta_{21}], \nabla_4=[\Delta_{22}],
\nabla_5=[\Delta_{23}-\Delta_{32}].\]

The automorphism group of ${\mathcal N}^{3*}_{04}(0)$ consists of invertible matrices of the form

\[\phi=\left(
                             \begin{array}{ccc}
                               x & 0 & 0   \\
                               0 & y & 0  \\
                               z & t & xy                               \end{array}\right)
                               .\]

Since
\[
\phi^T
                           \left(\begin{array}{ccc}
                                \alpha_1& 0 & \alpha_2  \\
                                 \alpha_3 & \alpha_4 & \alpha_5  \\
                                 0 & -\alpha_5 & 0 \\
                             \end{array}
                           \right)\phi
                           =\left(\begin{array}{ccc}
                                 x(x\alpha_1+z\alpha_2) & \alpha^* & x^2y\alpha_2   \\
                                 y(x\alpha_3 +z\alpha_5)  &
                                  y^2\alpha_4& xy^2\alpha_5  \\
                                 0 & -xy^2\alpha_5 & 0 \\
                             \end{array}\right),\]
we have that the action of $\operatorname{Aut} ({\mathcal N}^{3*}_{04}(0))$ on the subspace
$\langle  \sum\limits_{i=1}^5\alpha_i \nabla_i \rangle$
is given by
$\langle  \sum\limits_{i=1}^5\alpha^*_i \nabla_i \rangle,$ where

\[
\begin{array}{rcl}
\alpha^*_1&=&x(x\alpha_1+z\alpha_2);\\
\alpha^*_2&=&x^2y\alpha_2;\\
\alpha^*_3&=&y(x\alpha_3+z\alpha_5);\\
\alpha^*_4&=&y^2\alpha_4;\\
\alpha^*_5&=&xy^2\alpha_5.
\end{array}\]

It is easy to see that the elements $\alpha_1 \nabla_1 +\alpha_3\nabla_3 +\alpha_4\nabla_4$ give algebras
which are central extensions of $2$-dimensional algebras.
We find the following new cases:

\begin{enumerate}
\item $\alpha_4=\alpha_5=0, \alpha_2 \neq 0,$ then: 
\begin{enumerate}

    \item if $\alpha_3= 0,$ then choosing $z=-\frac{x\alpha_1}{\alpha_2}$, we have the representative
        $\langle \nabla_2 \rangle$.

    \item if $\alpha_3\neq0,$ then choosing $z=-\frac{x\alpha_1}{\alpha_2}$ and $x=\frac{\alpha_3}{\alpha_2},$
        we have the representative 
 $\langle \nabla_2 +\nabla_3\rangle$.

\end{enumerate}

\item $\alpha_4=\alpha_2=0, \alpha_5\neq 0,$ then:
\begin{enumerate}

    \item if $\alpha_1=0,$ then choosing $z=-\frac{x\alpha_3}{\alpha_5},$ we have the representative 
   $\langle \nabla_5\rangle$.

    \item if $\alpha_1\neq0,$ then choosing $z=-\frac{x\alpha_3}{\alpha_5}$ and $x=\frac{y^2\alpha_5}{\alpha_1},$ we have the representative 
   $\langle   \nabla_1+\nabla_5\rangle$.

\end{enumerate}

\item $\alpha_4=0, \alpha_5\neq0, \alpha_2\neq0,$ then:
\begin{enumerate}

 \item if  $\alpha_1\alpha_5-\alpha_2\alpha_3= 0,$ then choosing $z=-\frac{x\alpha_3}{\alpha_5}$ and $x=\frac{y\alpha_5}{\alpha_2},$ we have the representative
  $\langle \nabla_2+\nabla_5  \rangle$.

  \item if $\alpha_1\alpha_5-\alpha_2\alpha_3\neq 0,$ then choosing $z=-\frac{x\alpha_3}{\alpha_5},$
   $y=\frac{\alpha_1\alpha_5-\alpha_2\alpha_3}{\alpha_2\alpha_5}$ and
   $x=\frac{\alpha_1\alpha_5-\alpha_2\alpha_3}{\alpha_2^2},$  we have the representative
   $\langle \nabla_1+\nabla_2+\nabla_5 \rangle$.

\end{enumerate}

\item $\alpha_4\neq0, \alpha_5=0, \alpha_2\neq0,$ then:
\begin{enumerate}

 \item if  $\alpha_3= 0,$ then choosing $z=-\frac{x\alpha_1}{\alpha_2}$ and $y=\frac{x^2\alpha_2}{\alpha_4}$, we have the representative 
   $\langle
   \nabla_2+\nabla_4 \rangle$.

 \item if $\alpha_3\neq 0,$ then choosing $z=-\frac{x\alpha_1}{\alpha_2},$  $y=\frac{\alpha_3^2}{\alpha_2\alpha_4}$ and
   $x=\frac{\alpha_3}{\alpha_2}$,  we have the representative
   $\langle
   \nabla_2+\nabla_3+\nabla_4 \rangle$.

\end{enumerate}

\item $\alpha_4\neq0, \alpha_5\neq 0, \alpha_2=0,$ then:
\begin{enumerate}

 \item  if $\alpha_1= 0,$ then choosing $z=-\frac{x\alpha_3}{\alpha_5}$ and $x=\frac{\alpha_4}{\alpha_5}$, we have the representative
   $\langle \nabla_4+\nabla_5 \rangle$.

 \item if  $\alpha_1\neq 0,$ then choosing $z=-\frac{x\alpha_3}{\alpha_5},$  $y=\sqrt{\frac{\alpha_4\alpha_1}{\alpha_5^2}}$ and
   $x=\frac{\alpha_4}{\alpha_5}$,  we have the representative
   $\langle
   \nabla_1+\nabla_4+\nabla_5 \rangle$.

\end{enumerate}

\item $\alpha_4\neq0, \alpha_5\neq 0, \alpha_2\neq0,$ then choosing $z=-\frac{x\alpha_3}{\alpha_5},$  $y=\frac{\alpha_2\alpha_4}{\alpha_5^2}$,
   $x=\frac{\alpha_4}{\alpha_5}$ and
   $\alpha=\frac{\alpha_5(\alpha_1\alpha_5-\alpha_2\alpha_3)}{\alpha_2^2\alpha_4}$,  we have the representative 
   $\langle \alpha\nabla_1+\nabla_2+\nabla_4+\nabla_5 \rangle$.

\end{enumerate}

Now we have all the new $4$-dimensional nilpotent Novikov algebras constructed from  ${\mathcal
N}^{3*}_{04}:$
\[ {\mathcal N}^4_{10}, \ \ldots, {\mathcal N}^{4}_{20}(\alpha). \]
The multiplication tables of these algebras can be found in Appendix A.

\subsubsection{Central extensions of ${\mathcal N}^3_{01}$}


Let us use the following notations
\[ \nabla_1=[\Delta_{12}], \nabla_2=[\Delta_{13}]-[\Delta_{31}].\]

The automorphism group of ${\mathcal N}^3_{01}$ consists of invertible matrices of the form

\[\phi=\left(
                             \begin{array}{ccc}
                               x & 0 & 0   \\
                               y & x^2 & 0  \\
                               z & xy & x^3  \\
                             \end{array}
                           \right).\]

Since
\[
\phi^T
                           \left(\begin{array}{ccc}
                                0& \alpha_1 & \alpha_2  \\
                                 0 & 0 & 0  \\
                                 -\alpha_2 & 0 & 0 \\
                             \end{array}
                           \right)\phi
                           =
                             \left(\begin{array}{ccc}
                                 \alpha^* & x^3\alpha_1  +x^2y\alpha_2& x^4\alpha_2   \\
                                 \alpha^{**} & 0 & 0  \\
                                 -x^4\alpha_2  & 0 & 0 \\
                             \end{array}\right),\]

 we have that the action of $\operatorname{Aut} ({\mathcal N}^3_{01})$ on the subspace
$\Big \langle  \sum\limits_{i=1}^{2}\alpha_i \nabla_i \Big\rangle$
is given by
$\Big \langle  \sum\limits_{i=1}^{2}\alpha_i^* \nabla_i \Big\rangle,$
where
$$\begin{array}{rcl}
\alpha_1^* &=& x^3 \alpha_1 + x^2y\alpha_2\\  \
\alpha_2^* &=& x^4\alpha_2 .
\end{array}$$

Since the $2$-dimensional central extension of two dimensional algebras was already considered, we have $\alpha_2 \neq 0.$
Choosing $y=-\frac{\alpha_1y}{\alpha_2},$ $x=\frac{1}{\sqrt[4]{\alpha_2}},$  we have the representative  $\langle \nabla_2\rangle.$

Now we have the unique new $4$-dimensional nilpotent Novikov algebras constructed from    ${\mathcal N}^{3}_{01}:$
\[
{\mathcal N}^4_{21}.  \]

The multiplication table of this algebra can be found in Appendix A.

\subsubsection{Central extensions of ${\mathcal N}^3_{02}(\lambda)$}


Let us use the following notations
\[\nabla_1=[\Delta_{21}], \nabla_2=[(2-\lambda)\Delta_{13}+\lambda(\Delta_{22}+\Delta_{31})].\]

The automorphism group of ${\mathcal N}^3_{02}(\lambda)$ consists of invertible matrices of the form

\[\phi=\left(
                             \begin{array}{ccc}
                               x & 0 & 0   \\
                               y & x^2 & 0  \\
                               z & xy(1+\lambda) & x^3  \\
                             \end{array}\right)                       .\]

Since
\[
\phi^T
                           \left(\begin{array}{ccc}
                                0& 0 & (2-\lambda)\alpha_2  \\
                                 \alpha_1 & \lambda\alpha_2 & 0  \\
                                 \lambda\alpha_2 & 0 & 0 \\
                             \end{array}
                           \right)\phi
                           =
                             \left(\begin{array}{ccc}
                                 \alpha^* & (2+2\lambda-\lambda^2)x^2y\alpha_2 & (2-\lambda)x^4\alpha_2   \\
                                 x^3\alpha_1 + \lambda (2+\lambda) x^2y\alpha_2 & \lambda x^4\alpha_2 & 0  \\
                                 \lambda x^4\alpha_2 & 0 & 0 \\
                             \end{array}\right),\]

 we have that the action of $\operatorname{Aut} ({\mathcal N}^3_{02}(\lambda))$ on the subspace
$\langle  \sum\limits_{i=1}^{2} \alpha_i \nabla_i \rangle$
is given by
$\langle  \sum\limits_{i=1}^{2} \alpha_i^* \nabla_i \rangle,$
where

$$\begin{array}{rcl}
\alpha_1^* &=& x^3\alpha_1 + \lambda^2 (\lambda-1) x^2y\alpha_2 ,\\
\alpha^*_2 &=& x^4 \alpha_2 .
\end{array}$$

Since the $2$-dimensional central extension of two dimensional algebras was already considered, we have $\alpha_2 \neq 0.$ We have the following cases:

\begin{enumerate}
\item if $\lambda\neq 0, 1,$  then choosing $x=\frac{1}{\sqrt[4]{\alpha_2}},$ $y=-\frac{x\alpha_1}{\lambda^2(\lambda-1)\alpha_2},$  we have the representative $\langle \nabla_2\rangle.$

\item if $\lambda = 0$ or $\lambda=1,$ and $\alpha_1 = 0,$ then choosing
$x=\frac{\alpha_1}{\alpha_2},$ we have the representative $\langle \nabla_2\rangle.$

\item if $\lambda = 0$ or $\lambda=1,$ and $\alpha_1 \neq 0,$  then choosing
then choosing
$x=\frac{\alpha_1}{\alpha_2},$ we have the representative $\langle \nabla_1+\nabla_2\rangle.$

\end{enumerate}


Now we have all the new $4$-dimensional algebras constructed from  ${\mathcal N}^{3}_{02}(\lambda):$
\[
{\mathcal N}^4_{22}(\lambda), \ {\mathcal N}^4_{23}, \ {\mathcal N}^4_{24}. \]
The multiplication tables of these algebras can be found in Appendix A.

\subsection{The algebraic classification of $4$-dimensional nilpotent Novikov algebras}

Recall that the class of $n$-dimensional algebras defined by the identities $(xy)z=0$ and $x(yz)=0$
lies in the intersection of all the varieties of algebras defined by some family of polynomial identities of degree $3,$
such as Leibniz algebras, Zinbiel algebras or associative algebras.
On the other side,
every algebra defined by the identities $(xy)z=0$ and $x(yz)=0$ is a central extension of some suitable algebra with zero product.
The list of all  non-anticommutative $4$-dimensional algebras defined by the identities $(xy)z=0$ and $x(yz)=0$  can be found in \cite{demir}.
Note that there is only one $4$-dimensional nilpotent anticommutative algebra with identity  $(xy)z=0.$
Obviously every algebra from this list is a $4$-dimensional nilpotent "non-pure" Novikov algebra.
Our aim in the present part of the work is to find all $4$-dimensional nilpotent  "pure" Novikov algebras
which are different from the class of algebras defined by the identities $(xy)z=0$ and $x(yz)=0.$

Now we are ready to formulate the main result of this part of the paper.
The proof of the present theorem is  based on the classification of $3$-dimensional nilpotent Novikov algebras and the results of Section \ref{centrext}.

\begin{theorem}
Let $\mathcal N$ be a nonzero  $4$-dimensional nilpotent "pure" Novikov algebra over $\mathbb C.$
Then, $\mathcal N$ is isomorphic to one of the algebras listed in Table A (see Appendix).
\end{theorem}

\section{The geometric classification of nilpotent Novikov algebras}

\subsection{Definitions and notation}
Given an $n$-dimensional vector space $\mathbb V$, the set ${\rm Hom}(\mathbb V \otimes \mathbb V,\mathbb V) \cong \mathbb V^* \otimes \mathbb V^* \otimes \mathbb V$
is a vector space of dimension $n^3$. This space has the structure of the affine variety $\mathbb{C}^{n^3}.$ Indeed, let us fix a basis $e_1,\dots,e_n$ of $\mathbb V$. Then any $\mu\in {\rm Hom}(\mathbb V \otimes \mathbb V,\mathbb V)$ is determined by $n^3$ structure constants $c_{ij}^k\in\mathbb{C}$ such that
$\mu(e_i\otimes e_j)=\sum\limits_{k=1}^nc_{ij}^ke_k$. A subset of ${\rm Hom}(\mathbb V \otimes \mathbb V,\mathbb V)$ is {\it Zariski-closed} if it can be defined by a set of polynomial equations in the variables $c_{ij}^k$ ($1\le i,j,k\le n$).

Let $T$ be a set of polynomial identities.
Every algebra structure on $\mathbb V$ satisfying polynomial identities from $T$ forms a Zariski-closed subset of the variety ${\rm Hom}(\mathbb V \otimes \mathbb V,\mathbb V)$. We denote this subset by $\mathbb{L}(T)$.
The general linear group $GL(\mathbb V)$ acts on $\mathbb{L}(T)$ by conjugations:
$$ (g * \mu )(x\otimes y) = g\mu(g^{-1}x\otimes g^{-1}y)$$
for $x,y\in \mathbb V$, $\mu\in \mathbb{L}(T)\subset {\rm Hom}(\mathbb V \otimes\mathbb V, \mathbb V)$ and $g\in GL(\mathbb V)$.
Thus, $\mathbb{L}(T)$ is decomposed into $GL(\mathbb V)$-orbits that correspond to the isomorphism classes of algebras.
Let $O(\mu)$ denote the orbit of $\mu\in\mathbb{L}(T)$ under the action of $GL(\mathbb V)$ and $\overline{O(\mu)}$ denote the Zariski closure of $O(\mu)$.

Let $\mathcal A$ and $\mathcal B$ be two $n$-dimensional algebras satisfying the identities from $T$, and let $\mu,\lambda \in \mathbb{L}(T)$ represent $\mathcal A$ and $\mathcal B$, respectively.
We say that $\mathcal A$ degenerates to $\mathcal B$ and write $\mathcal A\to \mathcal B$ if $\lambda\in\overline{O(\mu)}$.
Note that in this case we have $\overline{O(\lambda)}\subset\overline{O(\mu)}$. Hence, the definition of a degeneration does not depend on the choice of $\mu$ and $\lambda$. If $\mathcal A\not\cong \mathcal B$, then the assertion $\mathcal A\to \mathcal B$ is called a {\it proper degeneration}. We write $\mathcal A\not\to \mathcal B$ if $\lambda\not\in\overline{O(\mu)}$.

Let $\mathcal A$ be represented by $\mu\in\mathbb{L}(T)$. Then  $\mathcal A$ is  {\it rigid} in $\mathbb{L}(T)$ if $O(\mu)$ is an open subset of $\mathbb{L}(T)$.
 Recall that a subset of a variety is called irreducible if it cannot be represented as a union of two non-trivial closed subsets.
 A maximal irreducible closed subset of a variety is called an {\it irreducible component}.
It is well known that any affine variety can be represented as a finite union of its irreducible components in a unique way.
The algebra $\mathcal A$ is rigid in $\mathbb{L}(T)$ if and only if $\overline{O(\mu)}$ is an irreducible component of $\mathbb{L}(T)$.

Given the spaces $U$ and $W$, we write simply $U>W$ instead of $\dim\,U>\dim\,W$.



\subsection{Method of the description of  degenerations of algebras}

In the present work we use the methods applied to Lie algebras in \cite{BC99,GRH,GRH2,S90}.
First of all, if $\mathcal A\to \mathcal B$ and $\mathcal A\not\cong \mathcal B$, then $\mathfrak{Der}(\mathcal A)<\mathfrak{Der}(\mathcal B)$, where $\mathfrak{Der}(\mathcal A)$ is the Lie algebra of derivations of $\mathcal A$. We will compute the dimensions of algebras of derivations and will check the assertion $\mathcal A\to \mathcal B$ only for such $\mathcal A$ and $\mathcal B$ that $\mathfrak{Der}(\mathcal A)<\mathfrak{Der}(\mathcal B)$.


To prove degenerations, we will construct families of matrices parametrized by $t$. Namely, let $\mathcal A$ and $\mathcal B$ be two algebras represented by the structures $\mu$ and $\lambda$ from $\mathbb{L}(T)$ respectively. Let $e_1,\dots, e_n$ be a basis of $\mathbb  V$ and $c_{ij}^k$ ($1\le i,j,k\le n$) be the structure constants of $\lambda$ in this basis. If there exist $a_i^j(t)\in\mathbb{C}$ ($1\le i,j\le n$, $t\in\mathbb{C}^*$) such that $E_i^t=\sum\limits_{j=1}^na_i^j(t)e_j$ ($1\le i\le n$) form a basis of $\mathbb V$ for any $t\in\mathbb{C}^*$, and the structure constants of $\mu$ in the basis $E_1^t,\dots, E_n^t$ are such polynomials $c_{ij}^k(t)\in\mathbb{C}[t]$ that $c_{ij}^k(0)=c_{ij}^k$, then $\mathcal A\to \mathcal B$. In this case  $E_1^t,\dots, E_n^t$ is called a {\it parametrized basis} for $\mathcal A\to \mathcal B$.

If the number of orbits under the action of $GL(\mathbb V)$ on  $\mathbb{L}(T)$ is finite, then the constructions of some degenerations and some non-degenerations give the description of all rigid algebras and irreducible components.
Since the variety of $4$-dimensional nilpotent Novikov algebras  contains infinitely many non-isomorphic algebras, we have to do some additional work.
Let $\mathcal A(*):=\{\mathcal A(\alpha)\}_{\alpha\in I}$ be a set of algebras, and let $\mathcal B$ be another algebra. Suppose that, for $\alpha\in I$, $\mathcal A(\alpha)$ is represented by the structure $\mu(\alpha)\in\mathbb{L}(T)$ and $B\in\mathbb{L}(T)$ is represented by the structure $\lambda$. Then $\mathcal A(*)\to \mathcal B$ means $\lambda\in\overline{\{O(\mu(\alpha))\}_{\alpha\in I}}$, and $\mathcal A(*)\not\to \mathcal B$ means $\lambda\not\in\overline{\{O(\mu(\alpha))\}_{\alpha\in I}}$.

Let $\mathcal A(*)$, $\mathcal B$, $\mu(\alpha)$ ($\alpha\in I$) and $\lambda$ be as above. To prove $\mathcal A(*)\to \mathcal B$ it is enough to construct a family of pairs $(f(t), g(t))$ parametrized by $t\in\mathbb{C}^*$, where $f(t)\in I$ and $g(t)\in GL(\mathbb V)$. Namely, let $e_1,\dots, e_n$ be a basis of $\mathbb V$ and $c_{ij}^k$ ($1\le i,j,k\le n$) be the structure constants of $\lambda$ in this basis. If we construct $a_i^j:\mathbb{C}^*\to \mathbb{C}$ ($1\le i,j\le n$) and $f: \mathbb{C}^* \to I$ such that $E_i^t=\sum\limits_{j=1}^na_i^j(t)e_j$ ($1\le i\le n$) form a basis of $\mathbb V$ for any  $t\in\mathbb{C}^*$, and the structure constants of $\mu_{f(t)}$ in the basis $E_1^t,\dots, E_n^t$ are such polynomials $c_{ij}^k(t)\in\mathbb{C}[t]$ that $c_{ij}^k(0)=c_{ij}^k$, then $\mathcal A(*)\to \mathcal B$. In this case  $E_1^t,\dots, E_n^t$ and $f(t)$ are called a parametrized basis and a {\it parametrized index} for $\mathcal A(*)\to \mathcal B$, respectively.

We now explain how to prove $\mathcal A(*)\not\to\mathcal  B$.
Note that if $\mathfrak{Der} \ \mathcal A(\alpha)  > \mathfrak{Der} \  \mathcal B$ for all $\alpha\in I$ then $\mathcal A(*)\not\to\mathcal B$.
One can use also the following generalization of Lemma from \cite{GRH}, whose proof is the same as the proof of Lemma.

\begin{lemma}\label{gmain}
Let $\mathfrak{B}$ be a Borel subgroup of $GL(\mathbb V)$ and $\mathcal{R}\subset \mathbb{L}(T)$ be a $\mathfrak{B}$-stable closed subset.
If $\mathcal A(*) \to \mathcal B$ and for any $\alpha\in I$ the algebra $\mathcal A(\alpha)$ can be represented by a structure $\mu(\alpha)\in\mathcal{R}$, then there is $\lambda\in \mathcal{R}$ representing $\mathcal B$.
\end{lemma}

\subsection{The geometric classification of $4$-dimensional nilpotent Novikov algebras}
The main result of the present section is the following theorem.

\begin{theorem}\label{geobl}
The variety of $4$-dimensional nilpotent Novikov algebras  has two  irreducible components
defined by  two infinite families of algebras ${\mathcal N}^4_{20}(\alpha)$ and ${\mathcal N}^4_{22}(\lambda).$
\end{theorem}

\begin{Proof}
Recall, that the full description of the degeneration system of $4$-dimensional trivial Novikov algebras was given in \cite{kppv}.
Using the cited result, we have that the variety of $4$-dimensional trivial Novikov algebras has two irreducible components given by the following
families of algebras:

$$\begin{array}{lllllll}
\mathfrak{N}_2(\alpha)  & e_1e_1 = e_3, &e_1e_2 = e_4,  &e_2e_1 = -\alpha e_3, &e_2e_2 = -e_4 \\

\mathfrak{N}_3(\alpha)  & e_1e_1 = e_4, &e_1e_2 = \alpha e_4,  &e_2e_1 = -\alpha e_4, &e_2e_2 = e_4,  &e_3e_3 = e_4.
\end{array}$$

Now we can prove that the variety of $4$-dimensional nilpotent Novikov algebras has two irreducible components.
By some easy calculations we have that
\[ \mathfrak{Der} \ {\mathcal N}^4_{20}(\alpha)=3,  \
   \mathfrak{Der} \ {\mathcal N}^4_{22}(\lambda)_{\lambda \neq 0,1}=3.\]

Since, the dimensions of derivations of these algebras are the smallest possible in this variety, then, we have that the families of algebras ${\mathcal N}^4_{20}(\alpha)$ and
${\mathcal N}^4_{22}(\lambda)$
give two irreducible components. 
The list of all necessary degenerations is given in Table B (see Appendix A).

\end{Proof}

\section*{Appendix A.}

 {\tiny

\[\begin{array}{lllllllllllll}

\multicolumn{8}{c}{ \mbox{ {\bf Table A.}
{\it The list of $4$-dimensional nilpotent "pure" Novikov algebras.}}} \\ \\

{\mathcal N}^4_{01} &:& e_1 e_1 = e_2  & e_2 e_1=e_3  \\
{\mathcal N}^4_{02}(\lambda) &:& e_1 e_1 = e_2 & e_1 e_2=e_3 & e_2 e_1=\lambda e_3   \\

{\mathcal N}^4_{03}  &:& e_1 e_1 = e_2 & e_1 e_2=e_4 & e_2 e_1=e_3   \\
{\mathcal N}^4_{04}(\alpha) &:&
e_1 e_1 = e_2, &   e_1e_2=e_4, & e_2e_1=\alpha e_4, & e_3e_3=e_4,  \\

{\mathcal N}^4_{05} &:&
e_1 e_1 = e_2, &   e_1e_2=e_4, & e_1e_3= e_4,&  e_2e_1= e_4, & e_3e_3=e_4 \\

{\mathcal N}^4_{06}(\alpha)_{\alpha\neq0} &:&
e_1 e_1 = e_2, &   e_1e_2=e_4, & e_1e_3=e_4, & e_2e_1=\alpha e_4  \\

{\mathcal N}^4_{07} &:&
e_1 e_1 = e_2, &    e_2e_1= e_4, & e_3e_3=e_4 \\

{\mathcal N}^4_{08} &:&
e_1 e_1 = e_2, & e_1e_3=e_4, &    e_2e_1= e_4  \\

{\mathcal N}^4_{09} &:&
e_1 e_1 = e_2, & e_1e_2=e_4, &    e_3e_1= e_4  \\

{\mathcal N}^{4}_{10} &:&     e_1 e_2=e_3  & e_1 e_3=e_4\\

{\mathcal N}^{4}_{11} &:&   e_1e_2=e_3 & e_1 e_3=e_4  & e_2 e_1=e_4, \\

{\mathcal N}^{4}_{12} &:&  e_1e_2=e_3 & e_2 e_3=e_4  & e_3 e_2=-e_4, \\

{\mathcal N}^{4}_{13} &:&  e_1e_2=e_3 & e_1 e_1=e_4 & e_2 e_3=e_4  & e_3 e_2=-e_4, \\

{\mathcal N}^{4}_{14} &:&  e_1e_2=e_3 & e_1 e_3=e_4 & e_2 e_3=e_4  & e_3 e_2=-e_4, \\

{\mathcal N}^{4}_{15} &:&  e_1e_2=e_3 & e_1 e_1=e_4 &  e_1 e_3=e_4 & e_2 e_3=e_4  & e_3 e_2=-e_4, \\

{\mathcal N}^{4}_{16} &:&  e_1e_2=e_3 &  e_1 e_3=e_4 & e_2 e_2=e_4,  \\

{\mathcal N}^{4}_{17} &:&  e_1e_2=e_3 &  e_1 e_3=e_4 & e_2 e_1=e_4 & e_2 e_2=e_4,  \\

{\mathcal N}^{4}_{18} &:&  e_1e_2=e_3 &  e_2 e_2=e_4 & e_2 e_3=e_4  & e_3 e_2=-e_4,  \\

{\mathcal N}^{4}_{19} &:&  e_1e_2=e_3 & e_1 e_1=e_4 &  e_2 e_2=e_4 & e_2 e_3=e_4  & e_3 e_2=-e_4,  \\

{\mathcal N}^{4}_{20}(\alpha) &:&  e_1e_2=e_3 & e_1 e_1=\alpha e_4 &  e_1 e_3=e_4 & e_2 e_2=e_4 & e_2
e_3=e_4  & e_3 e_2=-e_4,  \\

{\mathcal N}^4_{21} &:&
e_1 e_1 = e_2, &  e_2 e_1=e_3, &  e_1 e_3=e_4, & e_3 e_1=-e_4 \\

{\mathcal N}^4_{22}(\lambda) &:&
e_1 e_1 = e_2 &  e_1 e_2=e_3& e_1 e_3=(2-\lambda)e_4 &  e_2 e_1= \lambda e_3&  e_2 e_2=\lambda e_4&  e_3 e_1=\lambda e_4\\
{\mathcal N}^4_{23} &:&
e_1 e_1 = e_2 &  e_1 e_2=e_3 &  e_1 e_3=2e_4 & e_2 e_1= e_4  \\
{\mathcal N}^4_{24} &:&
e_1 e_1 = e_2 &  e_1 e_2=e_3 & e_1 e_3=e_4 &e_2 e_1= e_3 + e_4 & e_2 e_2= e_4&  e_3 e_1= e_4,

\end{array}\]

$$\begin{tabular}{|rcl|l|}

\multicolumn{4}{c}{ \mbox{ {\bf Table B.}
{\it Degenerations of Novikov algebras of dimension $4$.}}} \\
\multicolumn{4}{c}{}\\

  \hline
  ${\mathcal N}^{4}_{14} $&$\to$&$ {\mathcal N}^{4}_{01}$  & $E_1^t=t^{-1}(e_1-e_2),  E_2^t=-t^{-2}e_3,  E_3^t=-t^{-3}e_4,  E_4^t=-e_2$ \\
  \hline
  ${\mathcal N}^{4}_{22}(\lambda) $&$\to$&$ {\mathcal N}^{4}_{02}(\lambda)$    & $E_1^t=e_1,  E_2^t=e_2,  E_3^t=e_3, E_4^t=t^{-1}e_4$ \\
  \hline
  ${\mathcal N}^{4}_{23} $&$\to$&$ {\mathcal N}^{4}_{03}$   & $E_1^t=te_1,  E_2^t=t^2e_2,  E_3^t=t^3e_4,  E_4^t=t^3e_3$\\
  \hline
  ${\mathcal N}^{4}_{20}(-\frac{\beta}{(\beta+1)^2})  $&$\to$&$  {\mathcal N}^{4}_{04}(\beta)$   & $E_1^t=t^2(-\beta e_1+\frac{\beta^2}{\beta+1}e_2-\frac{\beta^2}{\beta+1}e_3 + \frac{\beta^3}{\beta+1}e_4)$ \\

&&&  $E_2^t=t^4(-\frac{\beta^3}{\beta+1}e_3+\frac{2\beta^4}{(\beta+1)^2}e_4),  E_3^t=t^3\frac{\beta^2}{\beta+1}e_2,  E_4^t=t^6\frac{\beta^4}{(\beta+1)^2}e_4$\\



\hline
    ${\mathcal N}^{4}_{20}(-\frac{1}{4}-\frac{1}{2}\sqrt[3]{\frac{t}{4}})  $&$\to$&$ {\mathcal N}^{4}_{05}$  &
$E^t_1=-\sqrt[3]{4t^2}e_1+\sqrt[3]{\frac{t^2}{2}}e_2-\sqrt[3]{\frac{t^2}{2}}e_3,
E^t_2=-\sqrt[3]{2t^4}e_3+(\sqrt[3]{2t^4}-\sqrt[3]{\frac{t^5}{2}})e_4,
E^t_3=te_2, E^t_4=t^2e_4$\\

\hline
    ${\mathcal N}^{4}_{04}(\alpha\neq 1) $&$\to$&$ {\mathcal N}^{4}_{06}(\alpha)$  & $E_1^t=t(e_1-\frac{\alpha}{(\alpha-1)^2}e_2+\frac{\alpha}{\alpha-1}e_3+\frac{\alpha^2}{(\alpha-1)^4}e_4),$
     \\
     &&&  $E_2^t=t^2(e_2-\frac{\alpha}{(\alpha-1)^2}e_4),  E_3^t=t^2(e_3-\frac{1}{\alpha-1}e_2+\frac{\alpha}{(\alpha-1)^3}e_4),  E_4^t=t^3e_4$ \\
   \hline
    ${\mathcal N}^{4}_{20}(0) $&$\to$&$ {\mathcal N}^{4}_{07}$  & $E_1^t=t^2(-e_1+e_2-e_3+e_4),  E_2^t=t^4(-e_3+2e_4),  E_3^t=-t^3e_2,  E_4^t=t^6e_4$ \\
 \hline
   ${\mathcal N}^{4}_{23} $&$\to$&$ {\mathcal N}^{4}_{08}$   & $E_1^t=te_1,  E_2^t=t^2e_2,  E_3^t=\frac{1}{2}t^2e_3,  E_4^t=t^3e_4$\\
  \hline
  ${\mathcal N}^{4}_{22}(2)  $&$\to$&$ {\mathcal N}^{4}_{09}$    & $E_1^t=t(e_1-e_2),  E_2^t=t^2(e_2-3e_3+2e_4),  E_3^t=t^2(e_3-4e_4), E_4^t=2t^3e_4$ \\
  \hline
    ${\mathcal N}^{4}_{11}  $&$\to$&$ {\mathcal N}^{4}_{10}$    & $E_1^t=t^{-1}e_1,  E_2^t=t^{-2}e_2,  E_3^t=t^{-3}e_3, E_4^t=t^{-4}e_4$ \\
  \hline
   ${\mathcal N}^{4}_{20}(t-1)  $&$\to$&$ {\mathcal N}^{4}_{11}$    & $E_1^t=t(e_1+e_3-e_4),  E_2^t=te_2,  E_3^t=t(e_3-e_4), E_4^t=te_4$ \\
  \hline${\mathcal N}^{4}_{13}  $&$\to$&$ {\mathcal N}^{4}_{12}$    & $E_1^t=te_1,  E_2^t=e_2,  E_3^t=te_3, E_4^t=te_4$ \\
  \hline ${\mathcal N}^{4}_{15}  $&$\to$&$ {\mathcal N}^{4}_{13}$    &  $E_1^t=t^2e_1,  E_2^t=te_2,  E_3^t=t^3e_3, E_4^t=t^4e_4$ \\
  \hline${\mathcal N}^{4}_{15}  $&$\to$&$ {\mathcal N}^{4}_{14}$    & $E_1^t=t^{-1}e_1,  E_2^t=t^{-1}e_2,  E_3^t=t^{-2}e_3, E_4^t=t^{-3}e_4$ \\
  \hline ${\mathcal N}^{4}_{20}(\frac{1}{t})  $&$\to$&$ {\mathcal N}^{4}_{15}$    & $E_1^t=t^{-1}e_1,  E_2^t=t^{-1}e_2,  E_3^t=t^{-2}e_3, E_4^t=t^{-3}e_4$ \\
  \hline ${\mathcal N}^{4}_{17} $&$\to$&$ {\mathcal N}^{4}_{16}$    & $E_1^t=t^{-1}e_1,  E_2^t=t^{-2}e_2,  E_3^t=t^{-3}e_3, E_4^t=t^{-4}e_4$ \\ \hline
  ${\mathcal N}^{4}_{20}(t^3-t) $&$\to$&$ {\mathcal N}^{4}_{17}$    & $E_1^t=t e_1+t^2e_3-t^3e_4,  E_2^t=t^2e_2,  E_3^t=t^3e_3-t^4e_4, E_4^t=t^{4}e_4$  \\ \hline
  ${\mathcal N}^{4}_{19} $&$\to$&$ {\mathcal N}^{4}_{18}$    & $E_1^t=e_1,  E_2^t=t^{-1}e_2,  E_3^t=t^{-1}e_3, E_4^t=t^{-2}e_4$ \\
  \hline ${\mathcal N}^{4}_{20}(\frac{1}{t^2}) $&$\to$&$ {\mathcal N}^{4}_{19}$    & $E_1^t=e_1,  E_2^t=t^{-1}e_2,  E_3^t=t^{-1}e_3, E_4^t=t^{-2}e_4$ \\

  \hline
    ${\mathcal N}^{4}_{22}(\frac 1 t) $&$\to$&$ {\mathcal N}^{4}_{21}$    & $E_1^t=e_1,  E_2^t=e_2,  E_3^t=\frac 1 t e_3, E_4^t= - \frac 1 {t^2} e_4$ \\
  \hline
  ${\mathcal N}^{4}_{22}(t) $&$\to$&$ {\mathcal N}^{4}_{23}$    & $E_1^t=e_1 + \frac{1} {t^2(t-1)}e_2,  E_2^t=e_2 + \frac{t+1} {t^2(t-1)}e_3+ \frac{1} {t^3(t-1)^2}e_4,  E_3^t=e_3+ \frac{2+2t-t^2} {t^2(t-1)}e_4, E_4^t=e_4$ \\
  \hline
${\mathcal N}^{4}_{22}(t+1) $&$\to$&$ {\mathcal N}^{4}_{24}$    & $E_1^t=e_1 + \frac{1} {t(t+1)^2}e_2,  E_2^t=e_2 + \frac{t+2} {t(t+1)^2}e_3+ \frac{1} {t^3(t+1)^2}e_4,  E_3^t=e_3+ \frac{3-t^2} {t(t+1)^2}e_4, E_4^t=e_4$ \\

  \hline
  ${\mathcal N}^{4}_{17} $&$\to$&$ \mathfrak{N}_2(\alpha)_{\alpha\neq 0, 1} $    & $E_1^t= - \frac{\sqrt{1-\alpha}}{\alpha} t e_1 + \frac{\sqrt{1-\alpha}}{\alpha} t e_2,  E_2^t=  \sqrt{1-\alpha} t e_1 + \frac{t}{\sqrt{1-\alpha}} e_3, E_3^t=  \frac{\alpha-1}{\alpha^2} t^2 e_3, E_4^t=-t^2 e_4$ \\
  \hline ${\mathcal N}^{4}_{07}(\alpha\neq -1) $&$\to$&$ \mathfrak{N}_3(i\frac{\alpha-1}{\alpha+1}) $    & $E_1^t=t\sqrt[3]{\frac{2}{\alpha+1}}  e_1,  E_2^t= i t \left(\sqrt[3]{\frac{2}{\alpha+1}}e_1 -  \sqrt[3]{\frac{2}{\alpha+1}}^2e_2 +  e_4\right),
 E_3^t=t e_3, E_4^t=t^2 e_4$ \\
  \hline

\end{tabular}$$
}

\end{document}